\documentclass[preprint,11pt]{elsarticle}

\usepackage{amsfonts, amsmath, amscd}
\usepackage[psamsfonts]{amssymb}

\usepackage{amssymb}

\usepackage{pb-diagram}

\usepackage[all,cmtip]{xy}

\usepackage[usenames]{color}

\newtheorem{theorem}{Theorem}[section]
\newtheorem{lemma}[theorem]{Lemma}
\newtheorem{corollary}[theorem]{Corollary}
\newtheorem{question}[theorem]{Question}
\newtheorem{remark}[theorem]{Remark}
\newtheorem{proposition}[theorem]{Proposition}

\newproof{proof}{Proof}

\numberwithin{equation}{section}

\newcommand{\CC}{C_k}
\newcommand{\NN}{\mathbb{N}}
\newcommand{\ZZ}{\mathbb{Z}}

\newcommand{\w}{\omega}

\newcommand{\KK}{\mathcal{K}}

\newcommand{\IR}{\mathbb{R}}

\newcommand{\e}{\varepsilon}
\newcommand{\cl}{\mathrm{cl}}

\renewcommand{\phi}{\varphi}

\newcommand{\supp}{\mathrm{supp}}

\newcommand{\TT}{\mathbb{T}}

\newcommand{\SI}{\underrightarrow{\mbox{\rm s-ind}}_n\,}

\input xy
\xyoption{all}

\begin{document}

\begin{frontmatter}

\title{Topological properties of spaces of Baire functions}

\author{S.~Gabriyelyan}
\ead{saak@math.bgu.ac.il}
\address{Department of Mathematics, Ben-Gurion University of the Negev, Beer-Sheva, P.O. 653, Israel}

\begin{abstract}
A fundamental result proved by Bourgain,  Fremlin and  Talagrand states that the space $B_1(M)$ of Baire one functions over a Polish space $M$ is an angelic space. Stegall extended this result by showing that the class $B_1(M,E)$ of Baire one functions valued in a normed space $E$ is angelic.
These results motivate our study of various topological properties in the classes $B_\alpha(X,G)$ of Baire-$\alpha$ functions, where $\alpha$ is a nonzero countable ordinal, $G$ is a metrizable non-precompact abelian group and $X$ is a $G$-Tychonoff first countable space. In particular, we show that (1) $B_\alpha(X,G)$ is a $\kappa$-Fr\'{e}chet--Urysohn space and hence it is an Ascoli space, and (2) $B_\alpha(X,G)$ is a $k$-space iff $X$ is countable.

\end{abstract}


\end{frontmatter}



\section{Introduction}


For Tychonoff (=completely refular and Hausdorff) spaces $X$ and $Y$, we denote by $C(X,Y)$ the family of all continuous functions from $X$ to $Y$. We say that  $X$ is {\em $Y$-Tychonoff} if for every closed subset $A$ of $X$, a point $x\in X\setminus A$, and two distinct points $y,z\in Y$ there exists $f\in C(X,Y)$ such that $f(x)=y$ and $f(A)=\{ z\}$, i.e., $x$ and $A$ are completely separated by a continuous function from $X$ to $Y$. In what follows we consider only $Y$-Tychonoff spaces to have the space $C(X,Y)$ sufficiently rich in the sense of the definition of $Y$-Tychonoff spaces. If $Y$ is an absolutely convex subset of a locally convex space, then every Tychonoff space $X$ is $Y$-Tychonoff (see Lemma \ref{l:E-Tychonoff} below).


The  spaces of Baire class $\alpha$ functions were defined and studied by Ren\'{e}~Baire in his PhD thesis \cite{Baire}. Let $Y$ be a Tychonoff space and let $X$ be a $Y$-Tychonoff space. The {\em Baire class zero} $B_0(X,Y)$ is the class $C(X,Y)$, and the {\em Baire class one} $B_1(X,Y)$ is the family of all functions from $X$ to $Y$ which are pointwise limits of sequences of continuous  functions. If $\alpha>1$ is a countable ordinal, the  {\em class $B_\alpha(X,Y)$ of Baire-$\alpha$ functions} is the family of all functions  $f:X\to Y$ such that  there exists a sequence $\{ f_n\}_{n\in\NN} \subseteq \bigcup_{i<\alpha} B_i(X,Y)$ which pointwise converges to $f$. The spaces $B_\alpha(X,Y)$ are endowed with the pointwise topology induced from the direct product $Y^X$. If $Y=\IR$, set $B_\alpha(X):=B_\alpha(X,\IR)$.

The most important case is the case when $X=M$ is a Polish space. The compact subsets of $B_1(M)$ (called Rosenthal compact) have been studied intensively by Rosenthal \cite{Rosenthal-compact}, Bourgain, Fremlin and Talagrand \cite{BFT}, Godefroy \cite{Godefroy}, Todor\v{c}evi\'{c} \cite{Todorcevic} and others. The following fundamental result is proved in \cite{BFT}.
\begin{theorem}[\cite{BFT}] \label{t:angelic-BFT}
If $X$ is a Polish space, then $B_1(M)$ is angelic.
\end{theorem}
Recall that a Hausdorff topological space $X$ is called an {\em angelic space} if (1) every relatively countably compact subset of $X$ is relatively compact, and (2) any compact subspace of $X$ is Fr\'{e}chet--Urysohn. In \cite[Corollary~7]{Stegall}, Stegall generalized Theorem \ref{t:angelic-BFT} as follows:
\begin{theorem}[\cite{Stegall}] \label{t:angelic-Stegall}
If $X$ is a Polish space and $G$ is a metric space, then $B_1(M,G)$ is angelic.
\end{theorem}
In particular, if $G$ is a Banach space, then  $B_1(M,G)$ is angelic. This last result was generalized by Mercourakis and Stamati in \cite[Theorem~1.8]{Mer-Sta}.


For the general case when $X$ is a Tychonoff space, it is known that the tightness of $B_1(X)$ is equal to $\sup_{n\in\NN} l\big(X^n_{\aleph_0}\big)$, where $l(Y)$ is the Lindel\"{o}f number of a space $Y$ and $X_{\aleph_0}$ is the $\aleph_0$-modification of $X$, see \cite{Pestryakov}. Lorch noted in \cite{Lorch} that $X_{\aleph_0}$ coincides with the minimal topology on $X$ generated by all functions of the first Baire class, which shows that usually the space $X_{\aleph_0}$ is sufficiently complicated. 
In \cite{Pytkeev-B1}, Pytkeev  showed that if $X$ is a \v{C}ech-complete Lindel\"{o}f space, then $B_1(X)$ is a $q$-space if and only if $X$ is perfectly normal. 

The aforementioned results motivate a more detailed study of topological properties of the spaces $B_\alpha(X,Y)$.
In this paper we concentrate mostly on the case when $Y=G$ is an abelian non-precompact metrizable group and $X$ is a $G$-Tychonoff first countable space. We examine the spaces $B_\alpha(X,G)$ in that case and  having one of the  topological properties described in the diagram below (which also shows relationships between the considered properties):
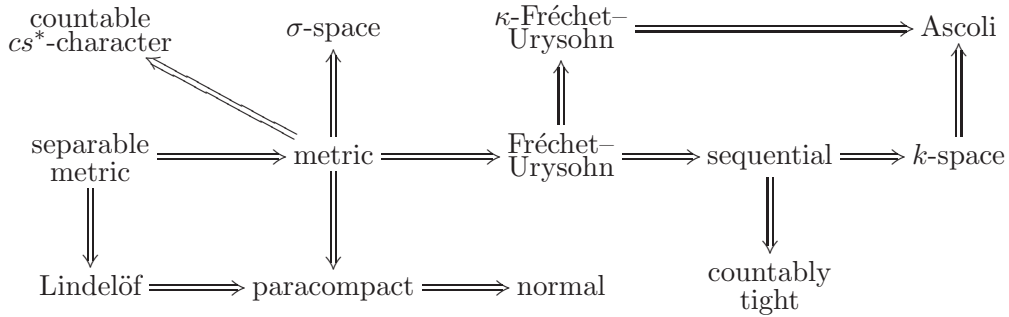
\[
\xymatrix{
{\substack{\mbox{countable} \\ \mbox{$cs^\ast$-character}}} & {\mbox{$\sigma$-space}} & {\substack{\mbox{$\kappa$-Fr\'{e}chet--} \\ \mbox{Urysohn}}} \ar@{=>}[rr] & & {\mbox{Ascoli}} \\
{\substack{\mbox{separable} \\ \mbox{metric}}}\ar@{=>}[r]  \ar@{=>}[d]  & {\mbox{metric}} \ar@{=>}[r] \ar@{=>}[u] \ar@{=>}[d] \ar@{=>}[lu] &   {\substack{\mbox{Fr\'{e}chet--} \\ \mbox{Urysohn}}} \ar@{=>}[r] \ar@{=>}[u] & {\mbox{sequential}} \ar@{=>}[r] \ar@{=>}[d] &  {\mbox{$k$-space}} \ar@{=>}[u] \\
{\mbox{Lindel\"{o}f}} \ar@{=>}[r]  & {\mbox{paracompact}} \ar@{=>}[r] & {\mbox{normal}} & {\substack{\mbox{countably} \\ \mbox{tight}}} &
}
\] %
Note that the implication ``$\kappa$-Fr\'{e}chet--Urysohn$\Rightarrow$Ascoli'' is proved in Theorem \ref{t:k-FU-seq-Ascoli} below,  other implications in the diagram are well known (all relevant definitions are given below in Sections \ref{sec:1} and \ref{sec:2}). 

Our main result is the following theorem.
\begin{theorem} \label{t:Baire-class}
Let $G$ be a non-precompact abelian metrizable group, $X$ a $G$-Tychonoff first countable space and let $H$ be a subgroup of $G^X$ containing $B_1(X,G)$. Then:
\begin{enumerate}
\item[{\rm (A)}] $H$ is a $\kappa$-Fr\'{e}chet--Urysohn space and hence is an Ascoli space.
\item[{\rm (B)}]  The following assertions are equivalent:
\begin{enumerate}
\item[{\rm (i)}] $X$ is countable;
\item[{\rm (ii)}] $H$ is a metrizable space and $H=G^X$;
\item[{\rm (iii)}] $H$ has countable tightness;
\item[{\rm (iv)}] $H$ has countable $cs^\ast$-character;
\item[{\rm (v)}] $H$ is a $\sigma$-space;
\item[{\rm (vi)}] $H$ is a $k$-space.
\end{enumerate}
If in addition $B_{2}(X,G)\subseteq H$, then (i)-(vi) are equivalent to
\begin{enumerate}
\item[{\rm (vii)}] $H$ is a normal space.
\end{enumerate}
\item[{\rm (C)}] If $B_{2}(X,G)\subseteq H$, then $H$ is a Lindel\"{o}f space if and only if $X$ is countable and $G$ is separable.
\item[{\rm (D)}]  $H$ is \v{C}ech-complete if and only if $X$ is countable and $G$ is complete.
\end{enumerate}
\end{theorem}
In Theorem \ref{t:Baire-class}, the assumption on $G$ of being non-precompact is essential for items (vi), (vii) and (C)-(D) (if $X$ is any discrete space and $G$ is compact, then $C(X,G)=G^X$ is a compact group). In Corollary \ref{c:Baire-class}  we consider the properties of being a locally compact, cosmic, analytic, or $K$-analytic space.

The paper is organized as follows. In Section \ref{sec:1} we give a new  characterization of $\kappa$-Fr\'{e}chet--Urysohn spaces  (Theorem \ref{t:Char-kFU}) which  implies $(A)$ of Theorem \ref{t:Baire-class}. In Theorem \ref{t:k-FU-seq-Ascoli}, we prove that every $\kappa$-Fr\'{e}chet--Urysohn space $X$ is Ascoli. Applying Theorem \ref{t:k-FU-seq-Ascoli} and some of the main results from \cite{Banakh-Survey,Gabr-LCS-Ascoli,Gab-LF,Gabr-L(X)-Ascoli,GKP} we characterize the $\kappa$-Fr\'echet--Urysohness in various important classes of locally convex spaces including strict $(LF)$-spaces and free locally convex spaces. In Section \ref{sec:2} we prove Theorem \ref{t:Baire-class} using several more general results. From that results it follows that (B)-(D) hold also for subspaces $H$ of $G^X$ containing $B_1(X,S)$ or $B_2(X,S)$, where $S$ is the closed unit ball of an infinite-dimensional normed space.


\section{The $\kappa$-Fr\'{e}chet--Urysohn property for locally convex spaces } \label{sec:1}


Following Arhangel'skii, a topological space $X$ is said to be {\em $\kappa$-Fr\'{e}chet--Urysohn} if for every open subset $U$ of $X$ and every $x\in \overline{U}$, there exists a sequence $\{ x_n\}_{n\in\NN} \subseteq U$ converging to $x$. Clearly, every Fr\'{e}chet--Urysohn space is $\kappa$-Fr\'{e}chet--Urysohn.  In \cite[Theorem~3.3]{LiL}, Liu and Ludwig showed that a topological space $X$ is $\kappa$-Fr\'{e}chet--Urysohn if and only if $X$ is a $\kappa$-pseudo open image of a metric space.  It is known that there are $\kappa$-Fr\'{e}chet--Urysohn spaces which are not $k$-spaces, and there are sequential spaces which are not $\kappa$-Fr\'{e}chet--Urysohn, see \cite{LiL} or Proposition \ref{p:kFU-phi} below.
In the next theorem we give a new characterization of $\kappa$-Fr\'{e}chet--Urysohn spaces. The closure of a subset $A$ of a topological space $X$ is denoted by $\overline{A}$ or $\cl_X(A)$.

\begin{theorem} \label{t:Char-kFU}
A topological space $X$ is $\kappa$-Fr\'{e}chet--Urysohn if and only if each point $x\in X$ is contained in a dense $\kappa$-Fr\'{e}chet--Urysohn subspace of $X$.
\end{theorem}

\begin{proof}
The necessity is clear. To prove the sufficiency, fix an open subset $U$ of $X$ and a point $x\in \overline{U}$. 
Let $Y$ be a dense $\kappa$-Fr\'{e}chet--Urysohn subspace of $X$ containing $x$. Then $V:= U\cap Y$ is an open subset of $Y$.
We claim that $x\in \cl_Y(V)$. Indeed, if $W\subseteq Y$ is an open neighborhood of $x$ in $Y$, take an open $W'\subseteq X$ such that $W=W'\cap Y$. As $x\in\overline{U}$, the set $W' \cap U$ is a nonempty open subset of $X$. Since $Y$ is dense in $X$ the set
$
(W'\cap U)\cap Y = (W'\cap Y) \cap (U\cap Y) = W\cap V
$
is not empty. Thus $x\in \cl_Y(V)$ and the claim is proved.
Finally, since $Y$ is $\kappa$-Fr\'{e}chet--Urysohn there is a sequence $\{ y_n\}_{n\in\NN} \subseteq V\subseteq U$ converging to $x$. \qed
\end{proof}

\begin{corollary} \label{c:kFU-dense-subgroup}
Let $Y$ be a dense subset of a homogeneous space (in particular, a topological group) $X$. If $Y$ is $\kappa$-Fr\'{e}chet--Urysohn, then $X$ is also a $\kappa$-Fr\'{e}chet--Urysohn space.
\end{corollary}
\begin{proof}
Fix an arbitrary $y_0\in Y$. Let $x\in X$. Take a homeomorphism $h$ of $X$ such that $h(y_0)=x$. Then $x\in h(Y)$ and $h(Y)$ is a $\kappa$-Fr\'{e}chet--Urysohn space. Therefore, each element of $X$ is contained in a dense $\kappa$-Fr\'{e}chet--Urysohn subspace of $X$ and Theorem \ref{t:Char-kFU} applies. \qed
\end{proof}


In \cite[Theorem~4.1]{LiL}, Liu and Ludwig proved that the product of a family of bi-sequential spaces is $\kappa$-Fr\'{e}chet--Urysohn. Note that any countable product of bi-sequential spaces is bi-sequential, see \cite[Proposition~3.D.3]{Michael-quotient}. On the other hand, countable products of $W$-spaces are $W$-spaces (\cite[Theorem~4.1]{Gruenhage-games}) and there are $W$-spaces which are not bi-sequential (\cite[Example~5.1]{Gruenhage-games}). Taking into account that bi-sequential spaces and $W$-spaces are Fr\'{e}chet--Urysohn spaces, the next corollary essentially generalizes Theorem 4.1 of \cite{LiL}.

\begin{corollary} \label{c:kFU-product}
Let $\{X_i:i\in I\}$ be a family of topological spaces such that $\prod_{i\in I'}X_i$ is Fr\'echet--Urysohn for any countable subset $I'$ of $I$. Then the space $X=\prod_{i\in I}X_i$ is $\kappa$-Fr\'{e}chet--Urysohn.
\end{corollary}

\begin{proof}
For every $z=(z_i)\in X$, set
\[
\sigma(z) := \big\{ x=(x_i)\in X: \{ i: x_i\not= z_i\} \mbox{ is finite}\big\}.
\]
Clearly, $\sigma(z)$ is a dense subspace of $X$. Proposition 2.6 of \cite{GGKZ} states 
 that $\sigma(z)$ is Fr\'echet--Urysohn. Thus, by  Theorem \ref{t:Char-kFU}, $X$ is $\kappa$-Fr\'{e}chet--Urysohn. \qed
\end{proof}

\begin{corollary} \label{c:kFU-metrizable-group}
Let $G$ be a nontrivial metrizable group with the identity $e$, $\kappa$ be a cardinal and let $H$ be a subgroup of the product $G^\kappa$ containing
\[
\bigoplus_\kappa G :=\big\{ f\in G^\kappa: \mathrm{supp}(f):= \{ i\in \kappa: f(i)\not= e\} \mbox{ is finite}\big\} .
\]
Then $H$ is a $\kappa$-Fr\'{e}chet--Urysohn space.
\end{corollary}

\begin{proof}
 Proposition 2.6 of \cite{GGKZ} implies 
that $\bigoplus_\kappa G$ is Fr\'echet--Urysohn. Clearly, $\bigoplus_\kappa G$ is dense in $G^\kappa$ and hence in $H$. Thus, by Corollary \ref{c:kFU-dense-subgroup}, $H$ is $\kappa$-Fr\'{e}chet--Urysohn. \qed
\end{proof}

Let $X$ be a Tychonoff  space. Denote by $\CC(X)$ and $C_p(X)$ the space $C(X)$ of all real-valued continuous functions on $X$ endowed with the compact-open topology and the pointwise topology, respectively.
Following \cite{BG}, $X$ is called an {\em Ascoli space} if every compact subset $\KK$ of $\CC(X)$ is evenly continuous (i.e., if the map $(f,x)\mapsto f(x)$ is continuous as a map from $\KK\times X$ to $\IR$). In \cite{Gabr-LCS-Ascoli}  we noticed that $X$ is Ascoli if and only if every compact subset of $\CC(X)$ is equicontinuous.  The classical Ascoli theorem \cite[Theorem~3.4.20]{Eng} states that every $k$-space is Ascoli. 
Now we prove the following somewhat unexpected result.

\begin{theorem} \label{t:k-FU-seq-Ascoli}
Each $\kappa$-Fr\'{e}chet--Urysohn Tychonoff space $X$ is Ascoli.
\end{theorem}

\begin{proof}
Suppose for a contradiction that $X$ is not an Ascoli space. Then there exists a compact set $K$ in $\CC(X)$ which is not equicontinuous at some point $z\in X$. Therefore there is $\e_0 >0$ such that for every open neighborhood $U$ of $z$ there exists a function $f_U\in K$ for which the open set $W_{f_U} :=\{ x\in U: | f_U (x)- f_U(z)|> \e_0\}$ is not empty (note that $z\not\in W_{f_U} \subseteq U$). Set
\[
W:= \bigcup\{ W_{f_U}: U \mbox{ is an open neighborhood of } z\}.
\]
Then $W$ is an open subset of $X$ such that $z\in\overline{W}\setminus W$. As $X$ is $\kappa$-Fr\'{e}chet--Urysohn, there is a sequence $\{ x_n :n\in\NN\} \subseteq W$ converging to $z$. For every $n\in\NN$, choose an open neighborhood $U_n$ of $z$ such that $x_n\in W_{f_{U_n}}$ and, therefore,
\begin{equation} \label{equ:kFU-Ascoli-1}
| f_{U_n} (x_n)- f_{U_n}(z)|> \e_0 \quad (\mbox{for all } n\in\NN).
\end{equation}
Set $S:= \{ x_{n}: n\in\NN\} \cup \{ z\}$. Then $S$ is a compact subset of $X$. Denote by $p$ the restriction map $p: \CC(X)\to \CC(S), p(f)=f|_S$. Then $p(K)$ is a compact subset of the Banach space $\CC(S)$. Applying the Ascoli theorem to the compact space $S$ we obtain that the sequence $\{ p(f_{U_n})\}_{n\in\NN}\subseteq p(K)$ is equicontinuous at $z\in S$ and, therefore, there is an $N\in\NN$ such that
\[
\big| f_{U_n} (x_{i}) - f_{U_n}(z) \big| <\frac{\e_0}{2} \; \; \mbox{ for all } \, i\geq N \, \mbox{ and } \, n\in\NN.
\]
In particular, for $i=n=N$ we obtain
$
\big| f_{U_N} (x_{N}) - f_{U_N}(z) \big| <\frac{\e_0}{2}.
$
But this contradicts (\ref{equ:kFU-Ascoli-1}). Thus $X$ is an Ascoli space. \qed
\end{proof}

In the rest of this section, using Theorem \ref{t:k-FU-seq-Ascoli} and some of the main results from \cite{Banakh-Survey,Gabr-LCS-Ascoli,Gab-LF,Gabr-L(X)-Ascoli,GKP},  we characterize the $\kappa$-Fr\'echet--Urysohness in various important classes of locally convex spaces.

In \cite[Theorem 2.1]{Sak2}, Sakai showed that $C_p(X)$ is $\kappa$-Fr\'{e}chet--Urysohn if and only if $X$ has the property $(\kappa)$. 
In \cite{GGKZ} we proved that if $C_p(X)$ is Ascoli,  then it is $\kappa$-Fr\'echet--Urysohn. These results and Theorem \ref{t:k-FU-seq-Ascoli} immediately imply
\begin{corollary} \label{t:Cp-seq-Ascoli}
Let $X$ be a Tychonoff space. Then $C_p(X)$ is Ascoli if and only if $X$ has the property $(\kappa)$.
\end{corollary}

The following corollary strengthens Theorem 1.3 of \cite{GGKZ}.
\begin{corollary} \label{t:Cech-complete-Cp-Ascoli}
Let $X$ be a \v{C}ech-complete space. Then $C_p(X)$ is Ascoli if and only if $X$ is scattered.
\end{corollary}

\begin{proof}
If $C_p(X)$ is Ascoli, then $X$ is scattered by Theorem 1.3 of \cite{GGKZ}. Conversely, if $X$ is scattered, then, by Corollary 3.8 of \cite{Sak2}, $X$ has the property $(\kappa)$. Thus, by Corollary \ref{t:Cp-seq-Ascoli}, $C_p(X)$ is Ascoli.\qed
\end{proof}

Let $E$ be a locally convex space over a field $\mathbf{F}$, where $\mathbf{F}=\IR$ or $\mathbb{C}$, and let $E'$ be the dual space of $E$. If $E$ is a Banach space, denote by $B$ the closed unit ball of $E$ and set
$
B_w:=\big(B,\sigma(E,E')|_B\big), 
$
where $\sigma(E,E')$ is the weak topology on $E$.

\begin{corollary} \label{c:Banach-kappa-FU}
{\rm (i)} If $E$ is a Banach space, then $B_w$ is $\kappa$-Fr\'{e}chet--Urysohn if and only if $E$ does not contain an isomorphic copy of $\ell_1$.

{\rm (ii)} A Fr\'{e}chet space $E$ over $\mathbf{F}$ is $\kappa$-Fr\'{e}chet--Urysohn in the weak topology if and only if $E=\mathbf{F}^N$ for some $N\leq\w$.

{\rm (iii)} If $X$ is a $\mu$-space and a $k_\IR$-space, then $\CC(X)$ is $\kappa$-Fr\'{e}chet--Urysohn in the weak topology if and only if $X$ is discrete.

{\rm (iv)} The weak${}^\ast$ dual space of a metrizable barrelled space $E$ is $\kappa$-Fr\'{e}chet--Urysohn if and only if $E$ is finite-dimensional.
\end{corollary}

\begin{proof}
(i) Theorem 1.9 of \cite{GKP} or Theorem 6.1.1 and Corollary 1.7 of \cite{Banakh-Survey} state that $B_w$ is Ascoli if and only if $B_w$ is Fr\'{e}chet--Urysohn if and only if $E$ does not contain an isomorphic copy of $\ell_1$. Now Theorem \ref{t:k-FU-seq-Ascoli} applies.

(ii) Corollary 1.7 of \cite{Gabr-LCS-Ascoli} states that $E$ is Ascoli in the weak topology if and only if $E=\mathbf{F}^N$ for some $N\leq\w$. This result and Theorem \ref{t:k-FU-seq-Ascoli} imply the desired.

(iii) Corollary 1.9 of \cite{Gabr-LCS-Ascoli} states that $\CC(X)$ is Ascoli in the weak topology if and only if $X$ is discrete. Now the assertion follows from Theorem \ref{t:k-FU-seq-Ascoli} and Corollary \ref{c:kFU-product}.

(iv) Corollary 1.14 of \cite{Gabr-LCS-Ascoli} states that the weak${}^\ast$ dual space of $E$ is Ascoli if and only if $E$ is finite-dimensional, and Theorem \ref{t:k-FU-seq-Ascoli} applies. \qed
\end{proof}

Now we consider direct locally convex sums of locally convex spaces. The simplest infinite direct sum of lcs is the space $\phi$, the direct locally convex sum $\bigoplus_{n\in\NN} E_n$ with $E_n=\mathbf{F}$ for all $n\in\NN$. It is well known that $\phi$ is a sequential non-Fr\'{e}chet--Urysohn space, see Example 1 of \cite{nyikos}. Below we strengthen the negative part of this assertion.

\begin{proposition}  \label{p:kFU-phi}
$\phi$ is a sequential non-$\kappa$-Fr\'{e}chet--Urysohn space.
\end{proposition}

\begin{proof}
The space $\phi$ is sequential by \cite[Example~1]{nyikos}. To show that $\phi$ is not $\kappa$-Fr\'{e}chet--Urysohn, we consider elements of $\phi$ as functions from $\NN$ to $\mathbf{F}$ with finite support. Recall that the sets of the form
\begin{equation} \label{equ:kFU-Ascoli-2}
\{ f\in \phi: |f(n)|<\e_n \mbox{ for every } n\in\NN\},
\end{equation}
where $\e_n >0$ for all $n\in\NN$, form a basis at $0$ of $\phi$ (see for example \cite[Example~1]{nyikos}).
For every $n,k\in\NN$, set
\[
U_{n,k}:= \left\{ f\in \phi: |f(1)|>\frac{1}{2n} \mbox{ and } |f(n)|>\frac{1}{2k} \right\},
\]
and set $U:= \bigcup_{n,k\in\NN} U_{n,k}$. It is easy to see that all the sets $U_{n,k}$ are open in $\phi$ and $0\not\in U_{n,k}$. Hence $U$ is an open subset of $\phi$ such that $0\not\in U$. To show that $\phi$ is not $\kappa$-Fr\'{e}chet--Urysohn, it suffices to prove that (A) $0\in \overline{U}$, and (B) there is no sequences in $U$ converging to $0$.

(A) Let $W$ be a basic neighborhood of zero in $\phi$ of the form (\ref{equ:kFU-Ascoli-2}). Choose an $n\in\NN$ such that $\frac{1}{n}<\e_1$, and take $k\in\NN$ such that $\frac{1}{k}<\e_n$. It is clear that $U_{n,k}\cap W$ is not empty. Thus $0\in \overline{U}$.

(B) Suppose for a contradiction that there is a sequence $S=\{ f_j\}_{j\in\NN}$ in $U$ converging to $0$. For every $j\in\NN$, take $n_j,k_j\in\NN$ such that $f_j \in U_{n_j,k_j}$. Since $f_j \to 0$, the definition of $U_{n,k}$ implies that $\frac{1}{2n_j}<|f_j(1)|\to 0$, and hence $n_j \to \infty$. Without loss of generality we can assume that $1<n_1 <n_2<\cdots$. For every $n\in\NN$, define $\e_n =\frac{1}{4k_j}$ if $n=n_j$ for some $j\in\NN$, and $\e_n =1$ otherwise. Set
\[
V:=\{f\in \phi: |f(n)|<\e_n  \mbox{ for every } n\in\NN \}.
\]
Then, $V$ is a neighborhood of $0$, and the construction of $U_{n,k}$ implies that $V\cap U_{n_j,k_j} =\emptyset$ for every $j\in\NN$. Thus $S\cap V=\emptyset$ and hence $f_j \not\to 0$, a contradiction. \qed
\end{proof}

\begin{corollary} \label{c:kFU-direct-sum}
An infinite direct sum of (non-trivial) locally convex spaces is not $\kappa$-Fr\'{e}chet--Urysohn.
\end{corollary}

\begin{proof}
Let $L=\bigoplus_{i\in I} E_i$ be the direct locally convex sum of an infinite family $\{ E_i\}_{i\in I}$ of locally convex spaces. It is well known 
that every $E_i$ can be represented as a direct sum $\mathbf{F}\oplus E'_i$. Therefore $L$ contains $\phi$ as a direct summand. Since the projection of $L$ onto $\phi$ is open and the $\kappa$-Fr\'{e}chet--Urysohn property is preserved under open maps (see Proposition 3.3 of \cite{LiL}), Proposition \ref{p:kFU-phi} implies that $L$ is not a $\kappa$-Fr\'{e}chet--Urysohn space.
\qed
\end{proof}

Recall that a {\em strict $(LF)$-space} $E$ is the direct limit $E=\SI E_n$ of an increasing sequence
\[
E_0 \hookrightarrow E_1 \hookrightarrow E_2 \hookrightarrow \cdots
\]
of Fr\'{e}chet (= locally convex complete metric linear) spaces in the category of locally convex spaces and continuous linear maps. The space $\mathcal{D}(\Omega)$ of test functions  over an open subset $\Omega$ of $\IR^n$ is one of the most famous and important examples of  strict $(LF)$-spaces which are not Fr\'{e}chet.
\begin{corollary} \label{t:kFU-strict-LF}
A strict $(LF)$-space $E$ is $\kappa$-Fr\'{e}chet--Urysohn if and only if $E$ is a Fr\'{e}chet space.
\end{corollary}

\begin{proof}
Theorem 1.2 
of \cite{Gab-LF} states that $E$ is an Ascoli space  if and only if $E$ is a Fr\'{e}chet space or $E=\phi$. Now the assertion follows from Theorem \ref{t:k-FU-seq-Ascoli} and Proposition \ref{p:kFU-phi}. \qed
\end{proof}
Consequently, $\mathcal{D}(\Omega)$ is not a $\kappa$-Fr\'{e}chet--Urysohn space.

One of the most important classes of locally convex spaces is the class of free locally convex spaces. Following \cite{Mar}, the {\em  free locally convex space}  $L(X)$ on a Tychonoff space $X$ is a pair consisting of a locally convex space $L(X)$ and  a continuous map $i: X\to L(X)$  such that every  continuous map $f$ from $X$ to a locally convex space  $E$ gives rise to a unique continuous linear operator ${\bar f}: L(X) \to E$  with $f={\bar f} \circ i$. The free locally convex space $L(X)$ always exists and is essentially unique.

\begin{corollary} \label{t:kFU-free-LCS}
Let $X$ be a Tychonoff space. Then $L(X)$ is a $\kappa$-Fr\'{e}chet--Urysohn space if and only if $X$ is finite.
\end{corollary}

\begin{proof}
It is well known that $L(D)$ over a countably infinite discrete space $D$ is topologically isomorphic to $\phi$.
By Theorem 1.2 
of \cite{Gabr-L(X)-Ascoli}, $L(X)$ is an Ascoli space if and only if $X$ is a countable discrete space. This fact, Theorem \ref{t:k-FU-seq-Ascoli} and Proposition \ref{p:kFU-phi} immediately imply   the assertion. \qed
\end{proof}


\section{Proof of Theorem \ref{t:Baire-class}} \label{sec:2}


We start from several lemmas in which we construct special functions from $B_1(X,G)$.
For every $g\in G$, define $\mathbf{g}:X\to G$ by $\mathbf{g}(x)=g$ for every $x\in X$.

\begin{lemma} \label{l:Y-Tychonoff}
Let $G$ be a nontrivial abelian (Hausdorff) topological group and let $X$ be a $G$-Tychonoff space. If $a_1,\dots,a_l\in X$ are distinct points, $U_{1}, \dots , U_{l}$ are pairwise disjoint open neighborhoods of $a_1,\dots,a_l$, respectively, and $g_0,g_1,\dots,g_l\in G$, then there is a continuous function $f:X\to G$ such that $f(X\setminus \bigcup_{i=1}^l U_i)=\{ g_0\}$ and $f(a_i)=g_i$ for every $i=1,\dots,l$.
\end{lemma}

\begin{proof}
Since $X$ is $G$-Tychonoff, for every $i=1,\dots,l$, there is a continuous function $f_i:X\to G$ such that $f_i(X\setminus U_i)=\{ 0\}$ and $f(a_i)=g_i -g_0$. Set $f:=f_1+\cdots+ f_l +\mathbf{g}_0$. Then $f$ is as desired. \qed
\end{proof}

\begin{corollary} \label{c:Baire-1}
Let $G$ be a nontrivial abelian group and let $X=\{ x_n\}_{n\in\NN}$ be a countable $G$-Tychonoff first countable space. Then $B_1(X,G)=G^X$.
\end{corollary}

\begin{proof}
Fix $f\in G^X$. For every $n\in\NN$, by Lemma \ref{l:Y-Tychonoff}, there is a continuous function $f_n:X\to G$ such that $f_n(x_k)=f(x_k)$ for every $k=1,\dots,n$. Then  $f_n \to f$ in $G^X$. Thus $f\in B_1(X,G)$. \qed
\end{proof}

Let $G$ be a nontrivial abelian topological group with zero $0$ and let $X$ be a $G$-Tychonoff space.
Recall that the sets of the form
\[
[F;V]:=\big\{ f\in G^X:  f(F)\subseteq V \big\},
\]
where $F$ is a finite subset of $X$ and $V$ is an open neighborhood of $0$, form a base of the pointwise topology on $G^X$ at zero function $\mathbf{0}$.
 For a function $f:X\to G$ and a subset  $A\subseteq G$, the set $\mathrm{supp}(f):=f^{-1}(G\setminus \{ 0\})$  is called the {\em support of} $f$ and define
\[
\begin{aligned}
\sigma(f) & := \big\{ h\in G^X: \{ x\in X: h(x)\not= f(x)\} \mbox{ is finite}\big\},\\
\sigma(f,A) & := \big\{ h\in \sigma(f): h(x)\in A \mbox{ for every } x\in X \mbox{ such that } h(x)\not= f(x)\big\}.
\end{aligned}
\]
Also we set $\sigma(f,g):=\sigma(f,\{ g\})$.

\begin{lemma} \label{l:E-Tychonoff}
Let $G$ be an absolutely convex subset of a locally convex space $E$, and let $X$ be a Tychonoff first countable space. Then:
\begin{enumerate}
\item[{\rm (i)}] $X$ is a $G$-Tychonoff space.
\item[{\rm (ii)}] Let $\{ U_n\}_{n< N}$, $1< N\leq\infty$, be a disjoint family of open subsets of $X$ and let $x_n\in U_n$ for every $n<N$.  Then, for every $g_1,g_2,\dots, g_N \in G$, the function
\[
f(x):= \left\{
\begin{aligned}
g_k, & \mbox{ if } x=x_k \mbox{ for some } k<N, \\
g_N, & \mbox{ if } x\in X\setminus \{ x_k: k<N\}
\end{aligned} \right.
\]
belongs to $B_1(X,G)$. In particular, $\sigma\big(\mathbf{g},G\big) \subseteq B_1(X,G)$ for every $g\in G$.
\end{enumerate}
\end{lemma}

\begin{proof}
(i) Let $x\in X$ and $A$ be a closed subset of $X$ such that $x\not\in A$. If $g,h$ are two distinct points in $G$, then the closed interval $[h,g]:=\{ h+\alpha(g-h): \alpha\in[0,1]\}\subseteq G$ is topologically isomorphic to $[0,1]$. As $X$ is Tychonoff, one can find a continuous function $f:X\to [h,g]$ such that $f(x)= g$ and $f(A)=\{ h\}$.

(ii) First we prove the following assertion which is similar to Lemma \ref{l:Y-Tychonoff}.

\smallskip
{\em Claim 1.  If $a_1,\dots,a_l\in X$ are distinct points, $U_{1}, \dots , U_{l}$ are pairwise disjoint open neighborhoods of $a_1,\dots,a_l$, respectively, and $g_0,g_1,\dots,g_l\in G$, then there is a continuous function $f:X\to G$ such that $f(X\setminus \bigcup_{i=1}^l U_i)=\{ g_0\}$ and $f(a_i)=g_i$ for every $i=1,\dots,l$.}

Indeed, for every $k=1,\dots,l$, by the proof of (i), there is a continuous function $f_k:X\to [0,g_k-g_0]$ such that $f_k(X\setminus U_k)=\{ 0\}$ and $f_k(a_k)=g_k -g_0$. Set $f:=f_1+\cdots+ f_l+\mathbf{g}_0$. We check that $\mathrm{Im}(f) \subseteq G$: if $x\in U_k$ for some $k=1,\dots,l$, we obtain $f_k(x)=\alpha(g_k-g_0)$ for some $\alpha\in[0,1]$ and hence
\[
f(x)= f_k(x) + \mathbf{g}_0(x)= \alpha(g_k-g_0) + g_0 =\alpha g_k +(1-\alpha)g_0 \in G,
\]
in particular, $f(a_k)=(g_k-g_0) + g_0=g_k$; if $x\not\in \bigcup_{k=1}^l U_k$, then $f(x)=\mathbf{g}_0(x)=g_0 \in G$. The claim is proved.

\smallskip
Now, for every $k<N$, choose a decreasing open base $\{ U_{n,k}\}_{n\in\NN}$ at $x_k$ such that $\overline{U_{1,k}}\subseteq U_k$.

\smallskip
{\em Case 1. Assume that $N$ is finite}. For every $n\in\NN$, by Claim 1, choose a continuous function $f_{n}:X\to G$ such that $f_n(X\setminus \bigcup_{k=1}^{N-1} U_{n,k})=\{ g_N\}$ and $f_n(x_k)=g_k$ for every $k=1,\dots,N-1$. It is clear that $f_n\to f$ in the pointwise topology. Thus $f\in B_1(X,G)$.

\smallskip
{\em Case 2. Assume that $N=\infty$}. For every $n\in\NN$, by Claim 1, choose a continuous function $f_{n}:X\to G$ such that $f_n(X\setminus \bigcup_{k=1}^{n} U_{n,k})=\{ g_\infty\}$ and $f_n(x_k)=g_k$ for every $k=1,\dots,n$. It is easy to see that $f_n\to f$ in the pointwise topology. Thus $f\in B_1(X,G)$.

Finally, the inclusion  $\sigma\big(\mathbf{g},G\big) \subseteq B_1(X,G)$ follows from the trivial fact that for every finite subset $\{ x_1,\dots,x_n\}$ of $X$ there is a disjoint family $\{ U_1,\dots,U_n\}$ of open sets of $X$ such that $x_i\in U_i$ for every $i=1,\dots,n$. \qed
\end{proof}

Using Lemma \ref{l:Y-Tychonoff} instead of Claim 1 in the proof of (ii) of Lemma \ref{l:E-Tychonoff} one can prove the following result.

\begin{lemma} \label{l:Baire-1}
Let $G$ be a nontrivial abelian topological group and let $X$ be an infinite $G$-Tychonoff first countable space. Let $\{ U_n\}_{n< N}$, $1< N\leq\infty$, be a disjoint family of open subsets of $X$ and let $x_n\in U_n$ for every $n<N$. Then, for every $g_1,g_2,\dots, g_N \in G$, the function
\[
f(x):= \left\{
\begin{aligned}
g_n, & \mbox{ if } x=x_n \mbox{ for some } n<N, \\
g_N, & \mbox{ if } x\in X\setminus \{ x_n: n<N\}
\end{aligned} \right.
\]
belongs to $B_1(X,G)$.  In particular, $\sigma\big(\mathbf{g}\big) \subseteq B_1(X,G)$ for every $g\in G$.
\end{lemma}



Let $E$ be a locally convex space (lcs for short) and let $X$ be a first countable Tychonoff space (so $X$ is $E$-Tychonoff, see Lemma \ref{l:E-Tychonoff}). A function $f:X\to E$ is called {\em bounded} if the image $\mathrm{Im}(f)$ of $f$ is a bounded subset of $E$.  If $S$ is a nontrivial absolutely convex subset of $E$, let $B_\alpha^b(X,S)$ be the family of all {\em bounded} functions from $B_\alpha(X,S)$.

\begin{theorem} \label{t:kFU-B1}
Let $G$ be a nontrivial abelian metrizable group and let $X$ be a $G$-Tychonoff first countable space. Assume that a subgroup $H$ of $G^X$ satisfies one of the following conditions:
\begin{enumerate}
\item[{\rm (a)}] $H$ contains $B_1(X,G)$;
\item[{\rm (b)}] $G$ is a metrizable lcs and $B_1^b(X,G)\subseteq H$.
\end{enumerate}
Then $H$ is a  $\kappa$-Fr\'{e}chet--Urysohn space and hence an Ascoli space.
\end{theorem}

\begin{proof}
By Lemma \ref{l:E-Tychonoff} and Lemma \ref{l:Baire-1}, we have $\sigma(\mathbf{0})\subseteq H$. Therefore, by Corollary \ref{c:kFU-metrizable-group},  $H$ is a $\kappa$-Fr\'{e}chet--Urysohn space. Thus, by Theorem \ref{t:k-FU-seq-Ascoli}, $H$ is an Ascoli space. \qed
\end{proof}

Recall that a Tychonoff space $X$
\begin{enumerate}
\item[$\bullet$] is a {\em $\sigma$-space} if $X$ has a $\sigma$-locally finite network (see \cite{gruenhage} for details);
\item[$\bullet$] is a {\em $k$-space} if for each non-closed subset $A\subset X$ there is a compact subset $K\subset X$ such that $K\cap A$ is not closed in $K$;
\item[$\bullet$] has  {\em  countable tightness} if whenever $A\subseteq X$ and $x\in \overline{A}$, then $x\in \overline{B}$ for some countable $B\subseteq A$;
\item[$\bullet$] has {\em countable $cs^\ast$-character} if $X$ has a countable $cs^\ast$-network at each point $x\in X$ (i.e., there is a countable family $\mathcal{N}_x$ of subsets of $X$ such that for each sequence $(x_n)_{n\in\NN}$ in $X$ converging to  $x$ and for each neighborhood $O_x$ of $x$ there is a set $N\in\mathcal{N}_x$ such that $x\in N\subseteq O_x$ and the set $\{n\in\NN :x_n\in N\}$ is infinite);
\end{enumerate}
It is well known that every compact subset of a $\sigma$-space is metrizable, see \cite{gruenhage}. Topological groups with countable $cs^\ast$-character are studied in \cite{BZ}. In \cite{GaK} we proved that a Baire topological vector space $E$ is metrizable if and only if $E$ has  countable $cs^\ast$-character, and the same metrizability condition holds also for $b$-Baire-like locally convex spaces.

\begin{theorem} \label{t:Baire-class-tightness}
Let $G$ be a nontrivial abelian metrizable group, $X$ be a $G$-Tychonoff first countable space and let $H$ be a subspace of $G^X$ containing $\sigma(\mathbf{0},g)\cup\{ \mathbf{g}\}$  for some nonzero $g\in G$. Then the following assertions are equivalent:
\begin{enumerate}
\item[{\rm (i)}] $X$ is countable;
\item[{\rm (ii)}] $H$ is a metrizable space;
\item[{\rm (iii)}] $H$ has countable tightness;
\item[{\rm (iv)}]  $H$ is a $\sigma$-space;
\item[{\rm (v)}] $H$ has countable $cs^\ast$-character.
\end{enumerate}
In particular, (i)-(v) are equivalent if $H$ satisfies one of the following conditions:
\begin{enumerate}
\item[{\rm (a)}] $H$ contains $B_1(X,G)$;
\item[{\rm (b)}] $G$ is a metrizable lcs and $B_1^b(X,S)\subseteq H$, where $S$ is a nontrivial absolutely convex subset of $G$.
\end{enumerate}
\end{theorem}

\begin{proof}
(i)$\Rightarrow$(ii) follows from the fact that $H$ is a subspace of the metrizable group $G^X$, and the implications (ii)$\Rightarrow$(iii)-(v) are clear.

(iii)$\Rightarrow$(i): Suppose for a contradiction that $X$ is uncountable. Fix an open neighborhood $V$ of $0\in G$ such that $0\not\in g+V$.
Set $A:= \sigma(\mathbf{0},g) \subseteq  H$. It is easy to see that $\mathbf{g}\in \overline{A}\setminus A$, where the closure is taken in $H$. Let $B$ be a countable subset of $A$. Set $D:= \bigcup \{ \supp(h): h\in B\}$. Then $D$ is a countable subset of $X$. As $X$ is uncountable there is $z\in X\setminus D$. Set $W:= \big[ \{z\}; V\big]$. Then $h(z)=0 \not\in g+V$ for every $h\in B$, and hence $(\mathbf{g}+W)\cap B=\emptyset$. So $\mathbf{g}\not\in \overline{B}$. Thus the tightness of $H$ is uncountable, a contradiction.

(iv)$\Rightarrow$(i): Suppose for a contradiction that $X$ is uncountable. Since every compact subset of a $\sigma$-space is metrizable (\cite[Corollary~4.7]{gruenhage}), it is sufficient to find a compact subset $K$ of $H$ which is not metrizable.  For every $x\in X$, define a function $\delta_{x,g}: X\to G$ by
\[
\delta_{x,g}(x)=g,\; \mbox{ and } \; \delta_{x,g}(y)=0 \mbox{ if } y\not= x.
\]
It is clear that $\delta_{x,g}\in \sigma(\mathbf{0},g)$. Set $K:=\{ \mathbf{0}\}\cup\{ \delta_{x,g}: x\in X\} \subseteq H$. Then any neighborhood of $\mathbf{0}$ contains all but finitely many of $\delta_{x,g}$s. Therefore $K$ is a compact subset of $H$. However, $K$ does not have countable base at $\mathbf{0}$ because $X$ is uncountable. This contradiction shows that $X$ must be countable. It is easy to see that $K$ is topologically isomorphic to the one point compactification of a discrete space of cardinality $|X|$.

(v)$\Rightarrow$(i):  Suppose for a contradiction that $X$ is uncountable. Consider the compact subset $K$ of $H$ defined in the previous paragraph. Then also $K$ has countable $cs^\ast$-character as a subspace of $H$. However, by Proposition 9 of \cite{BZ}, the  $cs^\ast$-character of $K$ is uncountable. This contradiction shows that $X$ is countable.

In the   cases (a)-(b), by Lemmas \ref{l:E-Tychonoff} and \ref{l:Baire-1}, there is a nonzero $g\in G$ or $g\in S$ such that $\sigma(\mathbf{0},g)\cup\{ \mathbf{g}\}\subseteq H$, and  the last assertion follows. \qed
\end{proof}

A sequence $\{ g_n\}_{n\in\NN}$ in an abelian topological group $G$ is called {\em uniformly discrete} if there is an open symmetric neighborhood $W$ of $0\in G$ such that
\begin{enumerate}
\item[$(\dag)$] $0\not\in g_k +(4)W$ for every $k\in\NN$, and
\item[$(\ddag)$] $g_k -g_n \not\in (4)W$ for all distinct $k,n\in\NN$.
\end{enumerate}

\begin{remark} \label{rem:B1-1} {\em
It follows from  Theorem 5 of \cite{BGP} that:

(i) any non-precompact abelian group $G$ contains a uniformly discrete sequence, and

(ii) if $G$ is an infinite-dimensional normed space, then the closed unit ball $S$ of $G$ contains a uniformly discrete sequence (otherwise, $S$ would be precompact and hence $G$ would be finite-dimensional).\qed }
\end{remark}

It is well known that  for every uncountable cardinal $\kappa$, the space $\ZZ^\kappa$ is neither a $k$-space nor a normal space. In the next two theorems we essentially generalize this result.

\begin{theorem} \label{t:Baire-class-k-space}
Let $G$ be an abelian metrizable group with zero $0$ containing a uniformly discrete sequence $\{ g_n\}_{n\in\NN}$, $X$ be a $G$-Tychonoff first countable space and let $H$ be a subspace of $G^X$ containing $\{ \mathbf{0}\} \cup \bigcup_{n\in \NN} \sigma(\mathbf{g}_n, 0)$. Then $H$ is a $k$-space if and only if $X$ is countable.
\end{theorem}

\begin{proof}
If $X$ is countable, then $H$ being a subspace of the metrizable group $G^X$ is metrizable, and hence $H$ is a $k$-space.  Assume that $H$ is a $k$-space and suppose for a contradiction that $X$ is uncountable.

Denote by $A$ the set of all functions $f:X\to G$ for which there is an $n\in\NN$ such that $\big|X\setminus\supp(f)\big|\leq n$ and $f(x)=g_n$ for every $x\in\supp(f)$ (so $f\in \sigma(\mathbf{g}_n,0)$). By assumption, $A\subseteq H$. Clearly, $\mathbf{0} \in \overline{A}\setminus A$, where the closure is taken in $H$. Therefore, to prove that $H$ is not a $k$-space it is sufficient to show that the set $A\cap K$ is closed in $K$ for every compact subset $K$ of $H$.
Set $B:= A\cap K$. We have to show that $\overline{B}=B$. Fix an arbitrary $f\in \overline{B}$, so $f\in K$.

{\em Claim 1. $f(X)\subseteq \{ 0, g_1,g_2, \dots\}$.} Indeed, fix an arbitrary $x\in X$. Then $(\dag)$-$(\ddag)$, the compactness of $K$ and the definition of $A$ imply that there is $k(x)\in\NN$ such that $h(x)\in \{ 0,g_1,\dots,g_{k(x)} \}$ for every $h\in A\cap K$. Therefore, also $f(x)\in \{ 0,g_1,\dots,g_{k(x)} \}$.

{\em Claim 2. $\sigma(\mathbf{0})\cap \overline{B}=\emptyset$.} Indeed, fix an arbitrary $t\in \sigma(\mathbf{0})$. Since $X$ is uncountable, there are $s\in\NN$ and an uncountable subset $Y$ of $X$ such that $k(x)=s$ for every $x\in Y$ (see Claim 1) and $Y\cap \supp(t)=\emptyset$. Let $C$ be a subset of $Y$ such that $|C|>2s$ (so $C\cap \supp(t)=\emptyset$). Then, by $(\dag)$-$(\ddag)$ and the definition of $A$, every $h\in (t+[C;W])\cap A$ must satisfy $h(x)=0$ for every $x\in C$. Now, the definition of $A$ implies that  $h(x)=g_k$ for some $k\geq |C|>2s$ and each $x\in \supp(h)$. As $Y\cap \supp(h)\not=\emptyset$, we obtain that $h(y)=g_k\not\in \{ 0,g_1,\dots,g_{s} \}$ for every $y\in Y\cap \supp(h)$. Therefore $h\not\in A\cap K$ and  $(t+[C;W])\cap (A\cap K)=\emptyset$. Thus $t\not\in \overline{B}$ and the claim is proved.

{\em Claim 3. There is an $m\in\NN$ such that $f(x)=g_m$ for every $x\in\supp(f)$.} Indeed, suppose for a contradiction that there are $y,z\in\supp(f)$ such that $f(y)\not=f(z)$. Since every $h\in A$ takes only one nonzero value, Claim 1 and $(\ddag)$ imply $\big(f+\big[ \{ y,z\}; W\big]\big)\cap A=\emptyset$. Thus $f\not\in \overline{A}$, a contradiction.

{\em Claim 4. If $C:= \{ x:f(x)=0\}$, then $|C| \leq m$.} Indeed, suppose for a contradiction that $|C| > m$. Choose a finite subset $C_0$ of $C$ such that $|C_0|>m$. Claim 2 implies that the support of $f$ is infinite. Therefore we can choose $x\in \supp(f)\subseteq X \setminus C_0$ (so $f(x)=g_m$ by Claim 3) and define
$
V:=\big[ \{x\}\cup C_0; W\big].
$
Then $(f+V)\cap H$ is a neighborhood of $f$ such that, by the definition of $A$ and  $(\dag)$-$(\ddag)$, the intersection $(f+V)\cap A$ is empty. Thus $f\not\in \overline{A}$, a contradiction.

Finally,  Claims 3 and 4 imply that $f\in A$, and hence $f\in B$. Thus $\overline{B}=B$.
\qed
\end{proof}

\begin{corollary} \label{c:Baire-class-k-space}
Let $S$ be an absolutely convex subset of a metrizable lcs $E$ containing a uniformly discrete sequence $\{ g_n\}_{n\in\NN}$, $X$ be a Tychonoff first countable space and let $H$ be a subspace of $E^X$ containing $B_1^b(X,S)$. Then $H$ is a $k$-space if and only if $X$ is countable.
\end{corollary}

\begin{proof}
By Lemma \ref{l:E-Tychonoff}, $\{ \mathbf{0}\}\cup \bigcup_{n\in \NN} \sigma(\mathbf{g}_n, 0)\subseteq B_1^b(X,S)\subseteq H$, and Theorem \ref{t:Baire-class-k-space} applies. \qed
\end{proof}
For example, Corollary  \ref{c:Baire-class-k-space} holds if $S$ is the closed unit ball of an infinite-dimensional normed space $E$, see   Remark \ref{rem:B1-1}.

\begin{theorem} \label{t:Baire-class-normal}
Let $G$ be an abelian metrizable group containing a uniformly discrete  sequence $\{ g_n\}_{n\in\NN}$. Set $T:=\{ 0\} \cup \{ g_n\}_{n\in\NN}$. Assume that  $X$ is a $G$-Tychonoff first countable space and let $H$ be a subspace of $G^X$ containing pointwise limits of sequences from $\sigma(\mathbf{g}_1, T)\cup \sigma(\mathbf{0},T)$. Then $H$ is a normal space if and only if $X$ is countable.
\end{theorem}

\begin{proof}
If $X$ is countable, then $H$ being a subgroup of the metrizable group $G^X$ is metrizable, and hence $H$ is a normal space.
Assume that $H$ is a normal space and suppose for a contradiction that $X$ is uncountable.

Denote by $D$ the family of all functions $f\in G^X$ such that
\begin{enumerate}
\item[(1)] $\supp(f)$ is finite;
\item[(2)] if $x\in \supp(f)$, there is $n\in\NN$ such that $f(x)=g_n$;
\item[(3)]  $|\{ x: f(x)=g_n\}| \leq 1$ for every  $n\in\NN$.
\end{enumerate}
By assumption, $D\subseteq \sigma(\mathbf{0},T)\subseteq H$. Set $A:=\cl_H(D)$.

{\em Claim 1. If $f\in A$, then}
\begin{enumerate}
\item[(4)]  $f\big(\supp(f)\big) \subseteq \{ g_n: n\in\NN\}$;
\item[(5)] if $f\not=\mathbf{0}$, then $|\{ x\in \supp(f): f(x)=g_n\}| \leq 1$ for every $n\in\NN$.
\end{enumerate}

Indeed, let $x\in \supp(f)$. Assuming that $f(x) \not\in \{ g_n: n\in\NN\}$, by $(\dag)$-$(\ddag)$, we can find an open neighborhood $W'\subseteq W$ of $0\in G$ such that $(f(x)+W') \cap  T=\emptyset$. Then, by (2),  $(f+[\{x\};W'])\cap D=\emptyset$. Thus $f\not\in A$. This contradiction proves (4). To prove (5), suppose that $f(y)=f(z)=g_n$ for some $n\in \NN$ and distinct $y,z\in \supp(f)$. Then, by $(\dag)$-$(\ddag)$ and (2)-(3),
$
\big( f+\big[ \{ y,z\}; W\big]\big) \cap D =\emptyset.
$
Hence $f\not\in A$, a contradiction. The claim is proved.

Set $g_0:=0$ and denote by $E$ the family of all functions $f\in G^X$ such that
\begin{enumerate}
\item[(6)] $\supp(f-\mathbf{g}_1)$ is finite;
\item[(7)] if $x\in \supp(f-\mathbf{g}_1)$, there is $n\in\{ 0,2,3,\dots\}$ such that $f(x)=g_n$;
\item[(8)]  $|\{ x: f(x)=g_n\}| \leq 1$ for every  $n\in\{ 0,2,3,\dots\}$.
\end{enumerate}
By assumption, $E\subseteq \sigma(\mathbf{g}_1,T)\subseteq H$. Set $B:=\cl_H(E)$.

{\em Claim 2. If $f\in B$, then}
\begin{enumerate}
\item[(9)]  $f\big(\supp(f-\mathbf{g}_1)\big) \subseteq \{ g_0,g_2,g_3,\dots\}$;
\item[(10)] if $f\not=\mathbf{g}_1$, then $|\{ x\in \supp(f): f(x)=g_n\}| \leq 1$ for every $n\in\{ 0,2,3,\dots\}$.
\end{enumerate}

We omit the proof of Claim 2 because it is similar to the proof of Claim 1.

\smallskip
{\em Claim 3. $A\cap B=\emptyset$.} Indeed, fix $f\in A$. It follows from (4) and (5) that $\supp(f)$ is countable. Since $X$ is uncountable, (9) and (10) imply that $f\not\in B$. Thus $A\cap B=\emptyset$, and the claim is proved.

\smallskip
As $H$ is normal, there are disjoint open sets $\mathcal{U},\mathcal{V} \subseteq H$ such that $A\subseteq \mathcal{U}$ and $B\subseteq \mathcal{V}$.

\smallskip
Now we define inductively a sequence $\{ f_n:n\in\NN\}\subseteq A$ as follows. Set $f_1:=\mathbf{0}$. Choose a finite subset $F_1\subseteq X$ and a neighborhood $W_1 \subseteq W$ of $0\in G$ such that $(f_1 +[F_1;W_1])\cap H \subseteq \mathcal{U}$. Assume that we found functions $f_1,\dots, f_n\in A$, finite sets
\[
F_1 =\{ x_1,\dots,x_{m(1)}\}, \dots, F_n=\{ x_1,\dots,x_{m(n)}\} \mbox{ with } m(1) < \cdots < m(n),
\]
and a decreasing sequence $W_1 \supseteq \cdots \supseteq W_n$ of open neighborhoods of zero in $G$ such that $(f_i+[F_i,W_i])\cap H \subseteq \mathcal{U}$ for every $i=1,\dots,n$. Define $f_{n+1}:X\to G$ by
\[
f_{n+1}(x):= \left\{
\begin{aligned}
g_k, & \mbox{ if } x=x_k \mbox{ for some } 1\leq k\leq m(n), \\
0, & \mbox{ if } x\in X\setminus \{ x_1,\dots,x_{m(n)}\}.
\end{aligned} \right.
\]
Then $f_{n+1}\in D\subseteq A$. So there are distinct points $x_{m(n)+1},\dots, x_{m(n+1)}\in X$ and an open neighborhood $W_{n+1} \subseteq W_n$ of $0\in G$ such that $(f_{n+1} + [F_{n+1};W_{n+1}])\cap H \subseteq \mathcal{U}$, where $F_{n+1} =\{ x_1,\dots,x_{m(n+1)}\}$. The induction is now complete.

For every $n\in\NN$, define $h_n\in E\subseteq B$ by
\[
h_{n}(x):= \left\{
\begin{aligned}
g_k, & \mbox{ if } x=x_k \mbox{ for some } 1\leq k\leq m(n), \\
g_1, & \mbox{ if } x\in X\setminus \{ x_1,\dots,x_{m(n)}\}.
\end{aligned} \right.
\]
Clearly, the sequence $\{h_n\}$ converges in $G^X$ to the function
\[
h(x):= \left\{
\begin{aligned}
g_{k}, & \mbox{ if } x=x_{k} \mbox{ for some } k\in\NN, \\
g_1, & \mbox{ if } x\in X\setminus \{ x_k\}_{k\in\NN}.
\end{aligned} \right.
\]
Therefore, by assumption,  $h\in H$ and hence $h\in B$.

Choose a finite subset $R$ of $X$ and an open neighborhood $W_0$ of zero in $G$ such that $(h+[R;W_0])\cap H \subseteq \mathcal{V}$. Since $R$ is finite, there is an $n\in\NN$ such that $\{ x_k\}_{k\in\NN} \cap R= F_n \cap R$. Now we define a function $t:X\to G$ by
\[
t(x):= \left\{
\begin{aligned}
g_{k}, & \mbox{ if } x=x_{k} \mbox{ for some } 1\leq k\leq m(n), \\
0, & \mbox{ if }  x\in F_{n+1}\setminus F_n =\{ x_{m(n)+1},\dots, x_{m(n+1)}\}, \\
g_1, & \mbox{ if } x\in X\setminus F_{n+1}.
\end{aligned} \right.
\]
Therefore $t\in \sigma(\mathbf{g}_1,T) \subseteq H$. By construction, $t(x)=h(x)$ for every $x\in R$, and hence
\[
t\in (h+[R;W_0])\cap H\subseteq \mathcal{V}.
\]
On the other hand, $t(x)= f_{n+1}(x)$ for every $x\in F_{n+1}$, and hence
\[
t\in (f_{n+1}+[F_{n+1};W_{n+1}])\cap H \subseteq \mathcal{U}.
\]
Therefore $t\in \mathcal{U}\cap \mathcal{V}=\emptyset$, a contradiction. Thus $X$ must be countable.\qed
\end{proof}

\begin{corollary} \label{c:Baire-1-3}
Let $E$ be an infinite-dimensional normed space, $X$  a first countable Tychonoff space and let $H$ contain $B_2^b(X,S)$, where $S$ is the closed unit ball of $E$. Then $H$ is a normal space if and only if $X$ is countable.
\end{corollary}

\begin{proof}
By Lemma \ref{l:E-Tychonoff}, $\sigma(\mathbf{g},S)\subseteq B_1^b(X,E)$ for every $g\in S$. By  Remark \ref{rem:B1-1}, $S$ contains a uniformly discrete sequence $T$ which is of course bounded. As $B_2^b(X,S)\subseteq H$, $H$ contains pointwise  limits of sequence from $\sigma(\mathbf{g}_1, T)\cup \sigma(\mathbf{0},T)$, where $g_1\in T$. Now Theorem \ref{t:Baire-class-normal} applies.\qed
\end{proof}

\begin{corollary} \label{c:Baire-class-Lindelof}
Let $G$ be a non-precompact abelian metrizable group, $X$ be a $G$-Tychonoff first countable space and let $H$ be a subspace of $G^X$ containing $B_2(X,G)$. Then the following assertions are equivalent:
\begin{enumerate}
\item[{\rm (i)}] $X$ is countable and $G$ is separable;
\item[{\rm (ii)}] $H$ is a Lindel\"{o}f space.
\end{enumerate}
\end{corollary}

\begin{proof}
(i)$\Rightarrow$(ii): If $X$ is countable and $G$ is separable, then $H$ is a subspace of the separable metrizable group $G^X$. Thus $H$ is a Lindel\"{o}f space.

(ii)$\Rightarrow$(i): The group $G$ contains a uniformly discrete sequence, see Remark \ref{rem:B1-1}. Therefore, by Lemma \ref{l:Baire-1} and Theorem \ref{t:Baire-class-normal}, the space $X$ is countable and hence, by Corollary \ref{c:Baire-1}, $H=G^X$. So also $G$ is Lindel\"{o}f. Being metrizable $G$ must be separable. \qed
\end{proof}

 Now we are ready to prove Theorem \ref{t:Baire-class}.

\medskip
{\em Proof of Theorem \ref{t:Baire-class}.}
(A) follows from Theorem \ref{t:kFU-B1}.

(B) The equivalence of (i)-(v) follows from Theorem \ref{t:Baire-class-tightness} and Corollary \ref{c:Baire-1}. (i) and (vi) are equivalent by Lemma \ref{l:Baire-1} and Theorem \ref{t:Baire-class-k-space}, and (i) and (vii) are equivalent by  Lemma \ref{l:Baire-1} and  Theorem \ref{t:Baire-class-normal}.

(C) follows from Corollary \ref{c:Baire-class-Lindelof}.

(D) Assume that $H$ is \v{C}ech-complete. Then $H$ is a $k$-space (\cite[Theorem~3.9.5]{Eng}), and (B) implies that $X$ is countable and $H=G^X$. Then $G$ being a closed subgroup of $G^X$ is also \v{C}ech-complete. Being metrizable $G$ must be complete. Conversely, assume that $X$ is countable and $G$ is complete. Then, by Corollary \ref{c:Baire-1}, $B_1(X,G)=G^X \subseteq H$. Hence $H=G^X$ is a complete metrizable group. Thus $H$ is \v{C}ech-complete.
\qed

\medskip


Recall that a Tychonoff space $X$ is
\begin{enumerate}
\item[$\bullet$] {\em cosmic } if it is a continuous image of a separable metrizable space;
\item[$\bullet$] {\em analytic} if it is a continuous image of a Polish space;
\item[$\bullet$] {\em $K$-analytic} if it is the image under usco compact-valued map defined on $\NN^\NN$.
\end{enumerate}
It is clear that every cosmic space is separable.
The next corollary completes Theorem \ref{t:Baire-class}.
\begin{corollary} \label{c:Baire-class}
Let $G$ be a non-precompact abelian metrizable group, $X$ a $G$-Tychonoff first countable space and let $H$ be a subgroup of $G^X$ containing $B_1(X,G)$. Then:
\begin{enumerate}
\item[{\rm (E)}] $H$ is locally compact if and only if $X$ is finite and $G$ is locally compact.
\item[{\rm (F)}]  $H$ is  cosmic if and only if $X$ is countable and $G$ is separable.
\item[{\rm (G)}] $H$ is analytic if and only if $X$ is countable and $G$ is analytic.
\item[{\rm (H)}] If $B_2(X,G)\subseteq H$, then $H$ is $K$-analytic if and only if $H$ is analytic if and only if  $X$ is countable and $G$ is analytic.
\end{enumerate}
\end{corollary}

\begin{proof}
(E) Assume that $H$ is locally compact. Then $H$ is \v{C}ech-complete, and hence, by (B) and (D) of Theorem \ref{t:Baire-class}, $X$ is countable, $G$ is complete and $H=G^X$. As $G$ is not compact, Theorem 3.3.13 of \cite{Eng} implies that $X$ is finite and $G$ is locally compact. The converse assertion is trivial.

(F) Assume that $H$ is cosmic. Then $H$ is a $\sigma$-space, and hence $X$ is countable and $H=G^X$, see (B) of  Theorem \ref{t:Baire-class}. Therefore, $G$ is also cosmic, and hence $G$ is separable. Conversely, if $X$ is countable and $G$ is separable, then, by Corollary \ref{c:Baire-1}, $H=G^X$. Therefore, $H$ is cosmic.

(G) Assume that $H$ is analytic. Then $H$ is cosmic and hence, by (F), $X$ is countable and $G$ is separable.  Corollary \ref{c:Baire-1} implies that $H=G^X$, and hence $G$ is analytic. Conversely, if $X$ is countable and $G$ is analytic, then,   by Corollary \ref{c:Baire-1}, $H=G^X$. Since the countable product of analytic spaces is analytic, we obtain that $H$ is an analytic space.

(H) Assume that $H$ is $K$-analytic. Then $H$ is Lindel\"{o}f, see Proposition 3.4 of \cite{kak}. Then (C) of Theorem \ref{t:Baire-class} implies that $X$ is countable and $G$ is separable. By Corollary \ref{c:Baire-1}, we obtain $H=G^X$. Thus, by Proposition 3.3 of \cite{kak}, $G$ is $K$-analytic. As $G$ is metrizable, it is analytic, see \cite[Proposition 6.3]{kak}. Conversely, if $X$ is countable and $G$ is analytic, then, by (G), $H$ is analytic and hence is $K$-analytic.
\qed
\end{proof}

We do not know whether the additional inclusion ``$B_2(X,G)\subseteq H$'' is essential in Theorem \ref{t:Baire-class} and Corollary \ref{c:Baire-class}.

\medskip
{\bf Acknowledge:} The author thanks Taras Banakh who brought to my attention that $K$-analytic metrizable spaces are analytic.

\medskip



\bibliographystyle{amsplain}

\end{document}